\documentclass[a4paper,11pt]{amsart}
\pdfoutput=1

\usepackage[utf8]{inputenc}
\usepackage{mathtools,amssymb}
\usepackage{hyperref}
\usepackage{xcolor,graphicx}
\usepackage{mathrsfs}
\usepackage{enumitem}

\graphicspath{{images/}}

\newtheorem{theorem}{Theorem}[section]
\newtheorem{lemma}[theorem]{Lemma}
\newtheorem{proposition}[theorem]{Proposition}

\newtheorem{remark}[theorem]{Remark}

\title[Partial data problems in scalar and vector field tomography]{Partial data problems and unique continuation in scalar and vector field tomography}

\author{Joonas Ilmavirta}
\thanks{Department of Mathematics and Statistics, University of Jyv\"askyl\"a, P.O. Box 35 (MaD) FI-40014 University of Jyv\"askyl\"a, Finland; \href{mailto:joonas.ilmavirta@jyu.fi}{joonas.ilmavirta@jyu.fi}; ORCID:0000-0002-2399-0911}
\author{Keijo M\"onkk\"onen}
\thanks{Department of Mathematics and Statistics, University of Jyv\"askyl\"a, P.O. Box 35 (MaD) FI-40014 University of Jyv\"askyl\"a, Finland; \href{mailto:keijo.m.t.monkkonen@jyu.fi}{keijo.m.t.monkkonen@jyu.fi}; ORCID:0000-0001-7848-8286}

\date{\today}

\usepackage{graphicx}

\newcommand{\C}{{\mathbb C}}
\newcommand{\R}{{\mathbb R}}
\newcommand{\Z}{{\mathbb Z}}
\newcommand{\N}{{\mathbb N}}

\newcommand{\der}{{\mathrm d}}
\newcommand*{\sol}[1]{#1^\mathrm{s}}

\newcommand{\xrt}{X}
\newcommand{\no}{N}
\newcommand{\dplane}{\mathcal{R}_d}
\newcommand{\nod}{\mathcal{N}_d}

\newcommand{\schwartz}{\mathscr{S}}
\newcommand{\tempered}{\mathscr{S}^{\prime}}
\newcommand{\rapidly}{\mathscr{O}_C^{\prime}}
\newcommand{\slowly}{\mathscr{O}_M}

\newcommand{\fourier}{\mathcal{F}}
\newcommand{\ifourier}{\mathcal{F}^{-1}}

\newcommand{\csmooth}{\mathcal{D}}
\newcommand{\smooth}{\mathcal{E}}
\newcommand{\cdistr}{\mathcal{E}'}
\newcommand{\distr}{\mathcal{D}^{\prime}}
\newcommand{\dimens}{n}
\newcommand{\fraclaplace}{(-\Delta)^s}
\newcommand{\pdos}{\mathcal{P}}
\newcommand{\adm}{\mathcal{A}}

\newcommand{\abs}[1]{\left\lvert #1 \right\rvert}
\newcommand{\aabs}[1]{\left\lVert #1 \right\rVert}
\newcommand{\ip}[2]{\left\langle #1,#2 \right\rangle}
\DeclareMathOperator{\spt}{spt}

\begin{document}

\maketitle

\begin{abstract}
We prove that if~$P(D)$ is some constant coefficient partial differential operator and~$f$ is a scalar field such that~$P(D)f$ vanishes in a given open set, then the integrals of~$f$ over all lines intersecting that open set determine the scalar field uniquely everywhere. This is done by proving a unique continuation property of fractional Laplacians which implies uniqueness for the partial data problem. We also apply our results to partial data problems of vector fields.
\end{abstract}

\section{Introduction}
Let~$f$ be a scalar field and $V\subset\R^n$ a nonempty open set where $n\geq 2$. We study the following partial data problem in X-ray tomography: can we say something about~$f$ if we know the integrals of~$f$ over all lines intersecting~$V$?
Especially, we are interested in the uniqueness problem which can be formulated in terms of the X-ray transform $\xrt_0$ as follows: if $\xrt_0 f=0$ on all lines which intersect~$V$, does it follow that $f=0$?
In general, the answer is ``no''~\cite{NA-mathematics-computerized-tomography} and one has to put some conditions on~$f|_V$.
We prove that if there is some constant coefficient partial differential operator~$P(D)$ such that $P(D)f|_V=0$ and $\xrt_0 f=0$ on all lines intersecting~$V$, then $f=0$.
This generalizes a recent partial data result in~\cite{IM-unique-continuation-riesz-potential}.
As a special case we obtain that if~$f$ is for example polynomial or (poly)harmonic in~$V$, then~$f$ is uniquely determined by its partial X-ray data. 

The partial data result is proved by using a unique continuation property of fractional Laplacian~$\fraclaplace$. We prove that if $s\in (-n/2, \infty)\setminus\Z$ and there is some constant coefficient partial differential operator~$P(D)$ such that $P(D)f|_V=\fraclaplace f|_V=0$, then $f=0$. This generalizes earlier results about unique continuation of fractional Laplacians~\cite{CMR-ucp-higher-order-laplacians, GSU-calderon-problem-fractional-schrodinger}. The unique continuation of~$\fraclaplace$ implies a unique continuation result for the normal operator~$\no_0$ of the X-ray transform~$\xrt_0$, and the uniqueness for the partial data problem follows directly from the unique continuation of~$\no_0$. This approach which uses the unique continuation of the normal operator in proving uniqueness for partial data problems was also used in~\cite{CMR-ucp-higher-order-laplacians, IM-unique-continuation-riesz-potential, IM-one-forms-partial-data}.

We also study partial data problems of vector fields.
Let~$F$ be a vector field and denote by~$\der F$ its exterior derivative or curl which components are $(\der F)_{ij}=\partial_i F_j-\partial_j F_i$.
We prove that if there are some constant coefficient partial differential operators~$P_{ij}(D)$ such that $P_{ij}(D)(\der F)_{ij}|_V=0$ and the integrals of~$F$ over all lines intersecting~$V$ vanish, then~$F$ must be a potential field ($F$ is the gradient of some scalar field).
This is a generalization of a recent result in~\cite{IM-one-forms-partial-data}.
The partial data result is proved by using a relation between the normal operator of the X-ray transform of scalar fields and the normal operator of the X-ray transform of vector fields (see lemma~\ref{lemma:scalarvectornormaloperator}).
This allows one to reduce the partial data problem for the vector field~$F$ to partial data problems for the scalar fields~$(\der F)_{ij}$.
As a special case analogous to the scalar result, we obtain that if~$F$ is for example componentwise polynomial or (poly)harmonic in~$V$, then the solenoidal part of~$F$ is uniquely determined by the partial X-ray data of~$F$.

The partial data problems we study have a relation to the region of interest (ROI) tomography~\cite{CNDK-solving-interior-problem-ct-with-apriori-knowledge, KKW-stability-of-interior-problems, KEQ-wavelet-methods-ROI-tomography, NA-mathematics-computerized-tomography, YYJW-high-order-TV-minimization}.
The main goal in such imaging problems is to determine the attenuation inside a small part of a human body (region of interest) by using only the X-ray data on lines which go through the ROI.
This for example reduces the needed X-ray dose which is given to the patient.
Our results imply that if the attenuation~$f$ satisfies $P(D)f|_V=0$ for some open subset~$V$ of the ROI and some constant coefficient partial differential operator~$P(D)$, then~$f$ is uniquely determined by its partial X-ray data on lines which intersect the ROI.
Note that~$f$ is uniquely determined not only in the ROI but also outside the ROI.
Concrete examples of admissible functions are listed in section~\ref{sec:admissible} below.
In general, $f$ does not have to be smooth and it can have singularities in the ROI.
We also note that our proof for uniqueness does not give stability for the partial data problem.
Especially, outside the ROI we have invisible singularities which cannot be seen by the X-ray data and the reconstruction of such singularities is not stable (see remark~\ref{remark:invisiblesingularities} and~\cite{KLM-on-local-tomography, QU-singularities-x-ray-transform-limited-data, QU-artifacts-and-singularities-limited-tomography}).

Similar ROI tomography problems can be studied in the case of vector fields. In vector field tomography the usual objective is to determine the velocity field of a fluid flow using acoustic travel time or Doppler backscattering measurements~\cite{NO-tomographic-recostruction-of-vector-fields, NO-unique-tomographic-reconstruction-doppler, SCHU-importance-of-vector-field-tomography}.
Assuming that the velocity of the fluid flow is much smaller than the speed of the propagating signal one can linearize the problem.
Linearization then leads to the X-ray transform of the velocity field.
Our results imply that if the velocity field~$F$ satisfies $P_{ij}(D)(\der F)_{ij}|_V=0$ for some open subset~$V$ of the ROI and some constant coefficient partial differential operators~$P_{ij}(D)$,
then the solenoidal part of~$F$ is uniquely determined everywhere by the partial X-ray data of~$F$ on lines intersecting the ROI.
The examples of admissible vector fields are the same as in the scalar case, only interpreted componentwise.
As in the scalar case,~$F$ can have singularities in the ROI, and our proof does not give stability for the partial data problem (since it is based on reduction to the scalar case).

The article is organized as follows. In section~\ref{sec:notation} we introduce our notation, in section~\ref{sec:admissible} we give examples of admissible functions, in section~\ref{sec:mainresults} we give our main theorems and in section~\ref{sec:relatedresults} we discuss some related results. We go through the theory of distributions and the X-ray transform in section~\ref{sec:preliminaries}, and study the space of admissible functions in section~\ref{sec:condition}. Finally, we prove our main results in section~\ref{sec:proofs}.

\subsection{Notation}
\label{sec:notation}
We quickly go through the notation used in our main theorems.
More detailed information about distributions and the X-ray transform of scalar and vector fields can be found in section~\ref{sec:preliminaries}.

We denote by~$f$ a scalar field. The set~$\rapidly(\R^n)$ is the space of rapidly decreasing distributions and the space~$\cdistr(\R^n)\subset\rapidly(\R^n)$ consists of compactly supported distributions.
The subset $C_\infty(\R^n)\subset\rapidly(\R^n)$ is the set of all continuous functions which decay faster than any polynomial at infinity.
We let~$\xrt_0$ be the X-ray transform of scalar fields and it maps a function to its line integrals (see equations~\eqref{eq:xrayscalar1} and~\eqref{eq:xrayscalar2}).
The normal operator is~$\no_0=\xrt_0^*\xrt_0$ where~$\xrt_0^*$ is the adjoint of~$\xrt_0$ (see equations~\eqref{eq:adjointscalar} and~\eqref{eq:normaloperatorscalar}).  

We denote by~$F$ a vector field.
The notation $F\in (\cdistr (\R^n))^n$ means that $F=(F_1, \dotso, F_n)$ where $F_i\in\cdistr(\R^n)$ for all~$i=1, \dotso, n$. The exterior derivative of~$F$ is written in components as $(\der F)_{ij}=\partial_i F_j-\partial_j F_i$.
For scalar fields~$\phi$ the notation $\der\phi$ denotes the gradient of~$\phi$.
We let~$\xrt_1$ be the X-ray transform of vector fields which maps a vector field to its line integrals (see equations~\eqref{eq:xrayvector1} and~\eqref{eq:xrayvector2}).
The normal operator is $\no_1=\xrt_1^*\xrt_1$ where~$\xrt_1^*$ is the adjoint of~$\xrt_1$ (see equations~\eqref{eq:adjointvector} and~\eqref{eq:normaloperatorvectorfield}).

We let $H^r(\R^n)$ be the fractional $L^2$-Sobolev space of order $r\in\R$ and $H^{-\infty}(\R^n)=\bigcup_{r\in\R}H^r(\R^n)$. We define the fractional Laplacian as~$\fraclaplace f=\ifourier(\abs{\cdot}^{2s}\hat{f})$ where $\hat{f}=\fourier(f)$ is the Fourier transform of~$f$ and~$\ifourier$ is the inverse Fourier transform. The fractional Laplacian~$\fraclaplace$ is well-defined in~$\rapidly(\R^n)$ for all $s\in (-n/2, \infty)\setminus \Z$ and in~$H^r(\R^n)$ for all $s\in (-n/4, \infty)\setminus \Z$. 

\subsection{Admissible functions}
\label{sec:admissible}

We denote by~$\pdos$ the set of all polynomials in~$\R^n$ with complex coefficients with the convention that the zero polynomial $P\equiv 0$ does not belong to~$\pdos$. A polynomial $P\in\pdos$ of degree~$m\in\N$ induces a constant coefficient partial differential operator~$P(D)$ of order $m\in\N$ by setting $P(D)=\sum_{|\alpha|\leq m}a_\alpha D^\alpha$ where $a_\alpha\in\C$, $D^\alpha=D_1^{\alpha_1}\cdots D_n^{\alpha_n}$, $D_j=-i\partial_j$ and $\alpha=(\alpha_1, \dotso, \alpha_n)\in\N^n$ is a multi-index such that $|\alpha|=\alpha_1+\dotso+\alpha_n$. The set of admissible functions~$\adm_V$ is defined as
\begin{align}
\label{eq:admissible}
\adm_V=\{f\in H^{-\infty}(\R^n): P(D)f|_V=0 \ \text{for some } P\in\pdos\}
\end{align}
where $V\subset\R^n$ is some nonempty open set. Examples of admissible functions include functions $f\in H^{-\infty}(\R^n)$ such that
\medskip
\begin{itemize}
    \item $f$ is polyharmonic in~$V$, i.e. $(-\Delta)^kf|_V=0$ for some $k\in\N$.
    \medskip
    \item $f$ is polynomial in~$V$.
    \medskip
     \item $f$ is of the form $f(x)=e^{ix\cdot\xi_0}$ in~$V$ where $\xi_0\in\C^n$ is a zero of $P\in\pdos$.\medskip
    \item $f$ is independent of one of the variables $x_1, \dotso , x_n$ in~$V$.
    \medskip
    \item $f$ satisfies the wave equation $(\partial_t^2-\Delta)f=0$ in $V$.
\end{itemize}

For convex sets~$V$ and a fixed $P\in\pdos$ the linear span of solutions of the form~$q(x)e^{ix\cdot\zeta}$ is dense in the space of all smooth solutions of $P(D)g=0$ in~$V$ (see~\cite[Theorem 7.3.6]{HO:analysis-of-pdos} and a more general result~\cite[Theorem 10.5.1]{HO:analysis-of-pdos2}).

If~$P(D)$ is a hypoelliptic operator, then the condition $P(D)f|_V=0$ already implies that~$f$ is smooth in~$V$ (see~\cite{HO:analysis-of-pdos2, MI:distribution-theory2}).
Basic examples of hypoelliptic operators are elliptic operators such as integer powers of Laplacians ($(-\Delta)^k$ where $k\in\N$) and also the non-elliptic heat operator $\partial_t-\Delta$.
However, there are non-smooth distributions~$f|_V$ which satisfy the condition $P(D)f|_V=0$ for some $P\in\pdos$ and therefore~$f$ can have singularities in~$V$. For example, the wave operator $\partial_t^2-\Delta$ is not hypoelliptic and has non-smooth weak solutions.
Another example of a non-hypoelliptic operator is the partial derivative~$\partial_{x_i}$: if~$f|_V$ is independent of~$x_i$, then the behaviour with respect to the other variables can be singular.

\subsection{Main results}
\label{sec:mainresults}
In this section we give our main theorems. The proofs of the results can be found in section~\ref{sec:proofs}. 

Our main theorem is the following unique continuation result for the fractional Laplacian.

\begin{theorem}
\label{thm:ucppolynomial}
Let $n\geq 1$, $s\in(-n/4, \infty)\setminus\Z$ and $f\in\adm_V$ where $V\subset\R^n$ is some nonempty open set. If $\fraclaplace f|_V=0$, then $f=0$. If $f\in\rapidly(\R^n)\cap\adm_V$, then the claim holds for $s\in(-n/2, \infty)\setminus\Z$.
\end{theorem}

Theorem~\ref{thm:ucppolynomial} generalizes the result in~\cite{CMR-ucp-higher-order-laplacians} (see lemma~\ref{lemma:uniquecontinuationoffractionallaplacian}) where one assumes that $\fraclaplace f|_V=f|_V=0$. In fact, theorem~\ref{thm:ucppolynomial} is proved by reducing the claim to the case treated in~\cite[Theorem 1.1]{CMR-ucp-higher-order-laplacians} (see section~\ref{sec:proofs}). The meaning of the condition $f\in\adm_V$ is discussed in section~\ref{sec:condition} (see remark~\ref{remark:condition}). When $s\in(-n/2, -n/4]\setminus\Z$, we need to have $f\in\rapidly(\R^n)$ so that~$\fraclaplace f$ is well-defined and we can use lemma~\ref{lemma:uniquecontinuationoffractionallaplacian} in the proof of theorem~\ref{thm:ucppolynomial}.

For compactly supported distributions we get a slightly stronger result.

\begin{theorem}
\label{thm:strongerucp}
Let $n\geq 2$, $s\in(-n/2, \infty)\setminus\Z$ and $f\in\cdistr(\R^n)\cap\adm_V$ where $V\subset\R^n$ is some nonempty open set. If $\partial^\beta(\fraclaplace f)(x_0)=0$ for some $x_0\in V$ and all $\beta\in\N^n$, then $f=0$.
\end{theorem}


From the unique continuation of fractional Laplacians we immediately obtain the following unique continuation result for the normal operator of the X-ray transform of scalar fields. The reason is that the normal operator can be written as $\no_0=(-\Delta)^{-1/2}$ up to a constant factor (see section~\ref{sec:xrayscalar}).

\begin{theorem}
\label{thm:ucpnormaloperator}
Let $n\geq 2$ and $f\in\cdistr(\R^n)\cap\adm_V$ or $f\in C_\infty(\R^n)\cap\adm_V$ where $V\subset\R^n$ is some nonempty open set. If $\no_0 f|_V=0$, then $f=0$.
\end{theorem}

Theorem~\ref{thm:ucpnormaloperator} is a generalization of the result in~\cite{IM-unique-continuation-riesz-potential} where one assumes $\no_0 f|_V=f|_V=0$. When $f\in\cdistr(\R^n)\cap\adm_V$, we could replace the assumption $\no_0 f|_V=0$ with the requirement that~$\no_0 f$ vanishes to infinite order at some point $x_0\in V$ (see theorem~\ref{thm:strongerucp}). In order to use theorem~\ref{thm:ucppolynomial} in the case $s=-1/2$ and $n\geq 2$, and to guarantee that $\no_0 f$ is well-defined, we need to have $f\in\cdistr(\R^n)\subset\rapidly(\R^n)$ or $f\in C_\infty(\R^n)\subset\rapidly(\R^n)$ in theorem~\ref{thm:ucpnormaloperator}. 

The unique continuation of~$\no_0$ implies uniqueness for the following partial data problem.

\begin{theorem}
\label{thm:xrayscalarpolynomial}
Let $n\geq 2$ and $f\in\cdistr(\R^n)\cap\adm_V$ or $f\in C_\infty(\R^n)\cap\adm_V$ where $V\subset\R^n$ is some nonempty open set. If $\xrt_0 f=0$ on all lines intersecting~$V$, then $f=0$.
\end{theorem}

Theorem~\ref{thm:xrayscalarpolynomial} generalizes theorem 1.2 in~\cite{IM-unique-continuation-riesz-potential}, where one assumes $f|_V=0$, to the case $P(D)f|_V=0$ for some $P\in\pdos$.
The case where~$f$ is polynomial in~$V$ is previously known in two dimensions~\cite{KKW-stability-of-interior-problems, YYJW-high-order-TV-minimization}.


It is important to notice that from the vector space structure of admissible functions~$\adm_V$ it follows that theorem~\ref{thm:xrayscalarpolynomial} is indeed a uniqueness result: if~$f_1$ and~$f_2$ satisfy $P_1(D)f_1|_V=P_2(D)f_2|_V=0$ for some $P_1, P_2\in\pdos$ and $\xrt_0 f_1=\xrt_0 f_2$ on all lines intersecting~$V$, then $f_1=f_2$ (see proposition~\ref{prop:vectorspace} and remark~\ref{remark:uniquenesspartialdata} for more details). 
Especially, the equality of the X-ray data on all lines intersecting~$V$ implies that the scalar fields are equal everywhere even though~$f_1$ and~$f_2$ a priori can have very different behaviour in~$V$ since~$P_1(D)$ can be different from~$P_2(D)$. 

\begin{remark}
\label{remark:invisiblesingularities}
Our proof for theorem~\ref{thm:xrayscalarpolynomial} gives only uniqueness but not stability for the partial data problem. In theorem~\ref{thm:xrayscalarpolynomial} we eventually have to assume that~$f$ is not supported in~$V$ since otherwise we would have $P(D)f=0$ everywhere and therefore~$f=0$ without assuming anything about the X-ray data (see the proof of theorem~\ref{thm:ucppolynomial}). When~$f$ is supported outside~$V$ we do not have access to all singularities of~$f$ via the X-ray data, i.e. we have invisible singularities outside~$V$. It is known that the recovery of such invisible singularities is not stable~\cite{KLM-on-local-tomography, QU-singularities-x-ray-transform-limited-data, QU-artifacts-and-singularities-limited-tomography}.
\end{remark}

\begin{remark}
\label{remark:dplanetransform}
Similar results as in theorems~\ref{thm:ucpnormaloperator} and~\ref{thm:xrayscalarpolynomial} also hold for the $d$-plane transform~$\dplane$ when~$d$ is odd (see corollaries 1 and 2 on page 646 in~\cite{CMR-ucp-higher-order-laplacians}). The $d$-plane transform~$\dplane$ takes a scalar field and integrates it over $d$-dimensional affine planes where $0<d<n$. The case $d=1$ corresponds to the X-ray transform and $d=n-1$ to the Radon transform. The normal operator of the $d$-plane transform is the composition~$\nod=\dplane^*\dplane$ where~$\dplane^*$ is the adjoint of~$\dplane$ and it can be expressed as $\nod=(-\Delta)^{-d/2}$ up to a constant factor (see~\cite{CMR-ucp-higher-order-laplacians, HE:integral-geometry-radon-transforms}). Hence~$\nod$ admits the same unique continuation property as in theorem~\ref{thm:ucppolynomial} for functions in $\cdistr(\R^n)\cap\adm_V$ or $C_\infty(\R^n)\cap\adm_V$ provided~$d$ is odd. The unique continuation of~$\nod$ then implies a similar uniqueness result as in theorem~\ref{thm:xrayscalarpolynomial} for a partial data problem of the $d$-plane transform~$\dplane$ when~$d$ is odd.
\end{remark}

From the unique continuation of fractional Laplacians we also obtain a partial data result for the X-ray transform of vector fields. The normal operators satisfy the relationship $\der(\no_1 F)=\no_0 (\der F)$ up to a constant factor (see lemma~\ref{lemma:scalarvectornormaloperator}). Hence the unique continuation and partial data problems of vector fields can be reduced to the corresponding problems for scalar fields, namely the components $(\der F)_{ij}$. 

The next theorems generalize the results in~\cite{IM-one-forms-partial-data} where the authors assume that $\der F|_V=0$ instead of $(\der F)_{ij}\in\adm_V$.
\begin{theorem}
\label{thm:ucpnormalvector}
Let $n\geq 2$ and $F\in (\cdistr(\R^n))^n$ such that $(\der F)_{ij}\in\adm_V$ for all $i, j=1, \dotso, n$ where $V\subset\R^n$ is some nonempty open set. If $\partial^\beta (\der(\no_1 F))(x_0)=0$ componentwise for some $x_0\in V$ and all $\beta\in\N^n$, then $F=\der\phi$ for some $\phi\in\cdistr(\R^n)$.
\end{theorem}

\begin{theorem}
\label{thm:xrayvectorpolynomial}
Let $n\geq 2$ and $F\in (\cdistr(\R^n))^n$ such that $(\der F)_{ij}\in\adm_V$ for all $i, j=1, \dotso, n$ where $V\subset\R^n$ is some nonempty open set. If $\xrt_1 F=0$ on all lines intersecting~$V$, then $F=\der\phi$ for some $\phi\in\cdistr(\R^n)$.
\end{theorem}

In light of the decomposition $F=\sol{F}+\der\phi$ of a vector field into a solenoidal part and a potential part, the conclusion $F=\der\phi$ of theorem~\ref{thm:xrayvectorpolynomial} can be recast as $\sol{F}=0$.
Therefore theorem~\ref{thm:xrayvectorpolynomial} can be seen as a solenoidal injectivity result in terms of partial data (see~\cite{IM-one-forms-partial-data} and~\cite{PSU-tensor-tomography-progress, SHA-integral-geometry-tensor-fields}).
Theorem~\ref{thm:xrayvectorpolynomial} holds also for vector fields $F\in(\schwartz(\R^n))^n$ which components are Schwartz functions since in that case $(\der F)_{ij}\in C_\infty(\R^n)\cap\adm_V$.

\subsection{Related results}
\label{sec:relatedresults}
There are some earlier unique continuation and partial data results for scalar and vector fields. The partial data problem for scalar fields has a unique solution if $f|_V$ vanishes~\cite{CNDK-solving-interior-problem-ct-with-apriori-knowledge, IM-unique-continuation-riesz-potential, KEQ-wavelet-methods-ROI-tomography}, $f|_V$ is polynomial or piecewise polynomial~\cite{KKW-stability-of-interior-problems, KEQ-wavelet-methods-ROI-tomography, YYJW-high-order-TV-minimization} or~$f|_V$ is real analytic~\cite{KKW-stability-of-interior-problems}.
A recent partial data result in two dimensions with attenuated X-ray data on an arc can be found in~\cite{FST-partial-inversion-2D}.
A complementary result is the Helgason support theorem: if the integrals of~$f$ vanish on all lines not intersecting a given compact and convex set, then~$f$ has to vanish outside that set~\cite{HE:integral-geometry-radon-transforms, SU:microlocal-analysis-integral-geometry}.
The normal operator of the X-ray transform of scalar fields has a unique continuation property under the assumptions $\no_0 f|_V=f|_V=0$ \cite{IM-unique-continuation-riesz-potential}.
This is a special case of a more general unique continuation property of fractional Laplacians~\cite{CMR-ucp-higher-order-laplacians, GSU-calderon-problem-fractional-schrodinger}.
There are also partial data and unique continuation results for the $d$-plane transform of scalar fields when~$d$ is odd, including the X-ray transform as a special case $d=1$ (see~\cite{CMR-ucp-higher-order-laplacians} and remark~\ref{remark:dplanetransform}).

The partial data problem of vector fields is known to be uniquely solvable up to potential fields, if $\der F|_V=0$~\cite{IM-one-forms-partial-data}. Similarly, the normal operator of the X-ray transform of vector fields has a unique continuation property under the assumptions $\no_1 F|_V=\der F|_V=0$~\cite{IM-one-forms-partial-data}. There are other partial data results for vector fields where one knows the integrals of~$F$ over lines which are parallel to a finite set of planes~\cite{JUH-principles-of-doppler-tomography, SCHU-3d-doppler-transform-reconstruction-and-kernels, SHA-vector-tomography-incomplete-data} or which intersect a certain type of curve~\cite{DEN-inversion-of-3d-tensor-fields, RA-microlocal-analysis-doppler-transform, VER-integral-geometry-symmetric-tensor-incomplete}. There is also a Helgason-type support theorem for vector fields: if the integrals of~$F$ vanish on all lines not intersecting a given compact and convex set, then~$\der F$ vanishes outside that set~\cite{IM-one-forms-partial-data, SU:microlocal-analysis-integral-geometry}. 

The normal operator of scalar fields, the normal operator of vector fields and the fractional Laplacian all admit stronger versions of the unique continuation property (see~\cite{CMR-ucp-higher-order-laplacians, FF-unique-continuation-fractional-ellliptic-equations, FF-unique-continuation-higher-laplacian, GR-fractional-laplacian-strong-unique-continuation, IM-unique-continuation-riesz-potential, IM-one-forms-partial-data, RU-unique-continuation-scrodinger-rough-potentials, YA-higher-order-laplacian} and theorems~\ref{thm:strongerucp} and~\ref{thm:ucpnormalvector}). Other applications of unique continuation of fractional Laplacians include fractional inverse problems. Especially, the unique continuation of~$\fraclaplace$ is used to prove uniqueness for different versions of the fractional Calder\'on problem (see e.g.~\cite{BGU-lower-order-nonlocal-perturbations, CLR18, CO-magnetic-fractional-schrodinger, CMR-ucp-higher-order-laplacians, CMRU-higher-order-fractional-perturbations, GSU-calderon-problem-fractional-schrodinger}).

\section{The X-ray transform and distributions}
\label{sec:preliminaries}
In this section we define the X-ray transform of scalar and vector fields, and introduce the distribution spaces we use in our main theorems. The basic theory of distributions and Sobolev spaces can be found in~\cite{GR-distributions-and-operators, HO:analysis-of-pdos, ML-strongly-elliptic-systems, MI:distribution-theory2, TRE:topological-vector-spaces-distributions} and the X-ray transform is treated for example in~\cite{NA-mathematics-computerized-tomography, SHA-integral-geometry-tensor-fields, SU:microlocal-analysis-integral-geometry}.

\subsection{Distributions and Sobolev spaces}
\label{sec:distributions}

The function spaces needed to state our theorems were described in section~\ref{sec:notation}.

We let~$\smooth(\R^n)$ be the space of smooth functions, $\schwartz(\R^n)$ is the Schwartz space and~$\csmooth(\R^n)$ is the space of compactly supported smooth functions. We equip all these spaces with their standard topologies. The corresponding duals are denoted by~$\cdistr(\R^n)$, $\tempered(\R^n)$ and~$\distr(\R^n)$. Elements in~$\cdistr(\R^n)$ are identified with distributions of compact support and elements in~$\tempered(\R^n)$ are called tempered distributions. 

We define the space of rapidly decreasing distributions~$\rapidly(\R^n)\subset\tempered(\R^n)$ as follows: $f\in\rapidly(\R^n)$ if and only if $\hat{f}\in\slowly(\R^n)$ where $\hat{f}=\fourier(f)$ is the Fourier transform of tempered distributions.
Here~$\slowly(\R^n)$ is the space of polynomially growing smooth functions, i.e.~$f\in\slowly(\R^n)$ if~$f$ and all its derivatives are polynomially bounded. 
We note that the Fourier transform is an isomorphism $\fourier\colon\tempered(\R^n)\rightarrow\tempered(\R^n)$ and also an isomorphism $\fourier\colon L^2(\R^n)\rightarrow L^2(\R^n)$.
We have the inclusions $\cdistr(\R^n)\subset\rapidly(\R^n)\subset\tempered(\R^n)\subset\distr(\R^n)$.
As a special case we have $\schwartz(\R^n)\subset C_\infty(\R^n)\subset\rapidly(\R^n)$ where~$C_\infty(\R^n)$ is the set of all continuous functions which decay faster than any polynomial at infinity. 

The fractional $L^2$-Sobolev space of order $r\in\R$ is defined as
\begin{equation}
H^r(\R^\dimens)=\{f\in\tempered(\R^\dimens): \langle\cdot\rangle^r\hat{f}\in L^2(\R^\dimens)\}
\end{equation}
where $\langle \xi\rangle=(1+\abs{\xi}^2)^{1/2}$. The space~$H^r(\R^n)$ is equipped with the norm
\begin{equation}
\aabs{f}_{H^{r}(\R^\dimens)}=\aabs{\langle\cdot\rangle^r\hat{f}}_{L^2(\R^\dimens)}
\end{equation}
and $H^r(\R^n)$ becomes a separable Hilbert space for every $r\in\R$.
It follows that the spaces are nested, i.e. $H^r(\R^n)\hookrightarrow H^t(\R^n)$ continuously when $r\geq t$.
One can isomorphically identify~$H^{-r}(\R^n)$ with the dual~$(H^r(\R^n))^*$ for all $r\in\R$.
We define the following spaces
\begin{equation}
H^\infty(\R^n)=\bigcap_{r\in\R}H^r(\R^n), \quad H^{-\infty}(\R^n)=\bigcup_{r\in\R}H^r(\R^n).
\end{equation}
It holds that $\rapidly(\R^n)\subset H^{-\infty}(\R^n)\subset\tempered(\R^n)$ and $\schwartz(\R^n)\subset H^\infty(\R^n)$. Further, using the Sobolev embedding one can see that $H^\infty(\R^n)=C^\infty_{L^2}(\R^n)$ where $f\in C^\infty_{L^2}(\R^n)$ if~$f$ is smooth and~$f$ and all its derivatives belong to~$L^2(\R^n)$ (see~\cite[Theorem 6.12]{GR-distributions-and-operators}).

The fractional Laplacian is defined as
\begin{equation}
\fraclaplace f=\ifourier(\abs{\cdot}^{2s}\hat{f})
\end{equation}
where~$\ifourier$ is the inverse Fourier transform of tempered distributions.
It follows that~$\fraclaplace f$ is well-defined as a tempered distribution for $f\in\rapidly(\R^n)$ when $s\in(-n/2, \infty)\setminus\Z$, and for $f\in H^r(\R^n)$ when $s\in(-n/4, \infty)\setminus\Z$ (see~\cite[Section 2.2]{CMR-ucp-higher-order-laplacians}). We have that $\fraclaplace\colon H^r(\R^n)\rightarrow H^{r-2s}(\R^n)$ is continuous whenever $s\in (0, \infty)\setminus\Z$ and~$\fraclaplace$ also admits a Poincar\'e-type inequality for $s\in (0, \infty)\setminus\Z$ (see~\cite{CMR-ucp-higher-order-laplacians}). We note that~$\fraclaplace$ is a non-local operator in contrast to the ordinary Laplacian~$(-\Delta)$. The non-locality implies a unique continuation property (see theorem~\ref{thm:ucppolynomial} and lemma~\ref{lemma:uniquecontinuationoffractionallaplacian}) which cannot hold for local operators.

We also use local versions of distributions and fractional Sobolev spaces. Let $\Omega\subset\R^n$ be an open set and $r\in\R$. We denote by $\csmooth(\Omega)$, $\distr(\Omega)$ etc. the test function and distribution spaces defined in~$\Omega$. We define the local Sobolev space~$H^r(\Omega)$ as 
\begin{align}
H^r(\Omega)=\{g\in\distr(\Omega): g=f|_\Omega \ \text{for some } f\in H^r(\R^n)\}.
\end{align}
In other words, the space~$H^r(\Omega)$ consists of restrictions of distributions $f\in H^r(\R^n)$. The local Sobolev space is equipped with the quotient norm
\begin{align}
\label{eq:quotientnorm}
\aabs{g}_{H^r(\Omega)}=\inf\{\aabs{f}_{H^r(\R^n)}: f\in H^r(\R^n) \ \text{such that } f|_\Omega=g\}.
\end{align}
Then $H^r(\Omega)$ becomes a separable Hilbert space and the restriction map $|_\Omega\colon H^r(\R^n)\rightarrow H^r(\Omega)$ is continuous.
If $r\geq t$, then $H^r(\Omega)\hookrightarrow H^t(\Omega)$ continuously. One can also isomorphically identify~$H^{-r}(\Omega)$ as the dual~$(\widetilde{H}^r(\Omega))^*$ for every $r\in\R$ where~$\widetilde{H}^r(\Omega)$ is the closure of~$\csmooth(\Omega)$ in~$H^r(\R^n)$ (see~\cite{CWHM-sobolev-spaces-on-non-lipchtiz-domains} and~\cite{ML-strongly-elliptic-systems}).
If $r\geq 0$, then $H^r(\Omega)\subset W^r(\Omega)$ where~$W^r(\Omega)$ is the Sobolev-Slobodeckij space which is defined by using weak derivatives of $L^2$-functions (see~\cite{ML-strongly-elliptic-systems} for a precise definition). If~$\Omega$ is a Lipschitz domain, then we have the equality $H^r(\Omega)=W^r(\Omega)$ for all $r\geq 0$.

More generally, we define the vector-valued test function space~$(\csmooth(\R^\dimens))^\dimens$ by saying that $\varphi\in (\csmooth(\R^\dimens))^\dimens$ if and only if $\varphi=(\varphi_1, \dotso, \varphi_\dimens)$ and $\varphi_i\in\csmooth(\R^\dimens)$ for all $i=1, \dotso, \dimens$. A sequence converges to zero in $(\csmooth(\R^\dimens))^\dimens$ if and only if all its components converge to zero in~$\csmooth(\R^n)$. We then define the space of vector-valued distributions~$(\distr(\R^\dimens))^\dimens$ by saying that $F\in(\distr(\R^\dimens))^\dimens$ if and only if $F=(F_1, \dotso, F_n)$ where $F_i\in\distr (\R^\dimens)$ for all $i=1, \dotso , \dimens$. The duality pairing is defined as $\ip{F}{\varphi}=\sum_{i=1}^\dimens\ip{F_i}{\varphi_i}$. The test function spaces~$(\smooth(\R^\dimens))^\dimens$ and $(\schwartz(\R^\dimens))^\dimens$, and the corresponding distribution spaces $(\cdistr(\R^\dimens))^\dimens$ and~$(\tempered(\R^\dimens))^\dimens$ are defined analogously. The elements in~$(\cdistr(\R^\dimens))^\dimens$ are called compactly supported vector-valued distributions. Vector-valued distributions are a special case of currents (continuous linear functionals in the space of differential forms, see~\cite[Section III]{deRham-differentiable-manifolds}). 

For $F\in(\distr(\R^n))^n$ we define the exterior derivative or curl of~$F$ as a matrix which components are $(\der F)_{ij}=\partial_i F_j-\partial_j F_i$. It follows from the Poincar\'e lemma (see e.g.~\cite[Theorem 2.1]{MA-poincare-derham-theorems} and lemma~\ref{lemma:poincarelemma}) that if $\der F=0$, then $F=\der\phi$ for some $\phi\in\distr(\R^n)$ where~$\der\phi$ is the distributional gradient of~$\phi$.

\subsection{The X-ray transform of scalar fields}
\label{sec:xrayscalar}

Let $f\in \csmooth(\R^n)$ be a scalar field.
The X-ray transform~$\xrt_0$ is defined as
\begin{equation}
\label{eq:xrayscalar1}
\xrt_0 f(\gamma)=\int_\gamma f\der s
\end{equation}
where~$\gamma$ is an oriented line in~$\R^n$.
When we parameterize the set of all oriented lines with the set
\begin{equation}
\label{eq:parametrizationlines}
\Gamma=\{(z,\theta): \theta\in S^{\dimens-1}, \ z\in\theta^{\perp}\}
\end{equation}
the X-ray transform becomes
\begin{equation}
\label{eq:xrayscalar2}
\xrt_0 f(z,\theta)=\int_{\R}f(z+s\theta)\der s,
\quad
f\in\csmooth(\R^n).
\end{equation}
The adjoint or back-projection~$\xrt_0^*$ is defined as
\begin{equation}
\label{eq:adjointscalar}
\xrt_0^*\psi(x)=\int_{S^{\dimens-1}}\psi(x-(x\cdot\theta)\theta, \theta)\der\theta,
\quad
\psi\in \smooth(\Gamma).
\end{equation}
One then sees that $\xrt_0\colon\csmooth(\R^\dimens)\rightarrow\csmooth(\Gamma)$ and $\xrt_0^*\colon\smooth(\Gamma)\rightarrow\smooth(\R^\dimens)$ are continuous maps.
Using duality we can define $\xrt_0\colon\cdistr(\R^\dimens)\rightarrow\cdistr(\Gamma)$ and $\xrt_0^*\colon\distr(\Gamma)\rightarrow\distr(\R^\dimens)$ by requiring that
\begin{align}
\langle \xrt_0 f, \varphi\rangle&=\langle f, \xrt_0^*\varphi\rangle, \quad f\in\cdistr(\R^n), \ \varphi\in\smooth(\Gamma) \\
\langle \xrt_0^*\psi, \eta\rangle&=\langle \psi, \xrt_0\eta\rangle, \quad \psi\in\distr(\Gamma), \ \eta\in\csmooth(\R^n),
\end{align}
where $\langle\cdot, \cdot\rangle$ is the dual pairing.

The normal operator is $\no_0=\xrt_0^*\xrt_0$ and it can be expressed as the convolution
\begin{equation}
\label{eq:normaloperatorscalar}
\no_0f(x)=2(f\ast\abs{\cdot}^{1-n})(x).
\end{equation}
Using duality the normal operator extends to a map $\no_0\colon\cdistr(\R^\dimens)\rightarrow\distr(\R^\dimens)$ and the convolution formula holds in the sense of distributions. The normal operator can be seen as the fractional Laplacian $(-\Delta)^{-1/2}$ up to a constant factor and we have the reconstruction formula
\begin{equation}
\label{eq:invertingnormaloperatorscalar}
f=c_{0, \dimens}(-\Delta)^{1/2}\no_0f
\end{equation}
where $c_{0, n}$ is a constant which depends on dimension. Both~$\xrt_0$ and~$\no_0$ are also defined for functions~$f\in C_\infty(\R^n)$.

\subsection{The X-ray transform of vector fields}
\label{sec:xrayvector}
Let $F\in (\csmooth(\R^n))^n$ be a vector field. The X-ray transform~$\xrt_1$ is defined as
\begin{equation}
\label{eq:xrayvector1}
\xrt_1 F(\gamma)=\int_\gamma F\cdot\der\overline{s}
\end{equation}
where $\gamma$ is an oriented line. Using the parametrization~$\Gamma$ for oriented lines (see equation~\eqref{eq:parametrizationlines}) we have
\begin{equation}
\label{eq:xrayvector2}
\xrt_1 F(z, \theta)=\int_{\R}F(z+s\theta)\cdot\theta\der s, \quad F\in (\csmooth(\R^n))^n.
\end{equation}
We define the adjoint~$\xrt_1^*$ as the vector-valued operator
\begin{equation}
\label{eq:adjointvector}
(\xrt_1^*\psi)_i(x)=\int_{S^{\dimens-1}}\theta_i \psi(x-(x\cdot\theta)\theta, \theta)\der\theta, \quad \psi\in \smooth(\Gamma).
\end{equation}
One sees that $\xrt_1\colon(\csmooth(\R^\dimens))^\dimens\rightarrow\csmooth(\Gamma)$ and $\xrt_1^*\colon\smooth(\Gamma)\rightarrow(\smooth(\R^\dimens))^\dimens$ are continuous and by duality we can define $\xrt_1\colon(\cdistr(\R^\dimens))^\dimens\rightarrow\cdistr(\Gamma)$ and $\xrt_1^*\colon\distr(\Gamma)\rightarrow(\distr(\R^\dimens))^\dimens$ by setting
\begin{align}
\ip{\xrt_1 F}{\varphi}&=\ip{F}{\xrt_1^*\varphi}, \quad F\in(\cdistr(\R^n))^n, \ \varphi\in\smooth(\Gamma) \\
\ip{\xrt_1^* \psi}{\eta}&=\ip{\psi}{\xrt_1\eta}, \quad \psi\in\distr(\Gamma), \ \eta\in(\csmooth(\R^n))^n.
\end{align} 

We define the normal operator as $\no_1=\xrt_1^*\xrt_1$ and it can be expressed in terms of convolution
\begin{equation}
\label{eq:normaloperatorvectorfield}
(\no_1 F)_i=\sum_{j=1}^\dimens\frac{2x_ix_j}{\abs{x}^{\dimens+1}}\ast F_j.
\end{equation}
The normal operator extends to a map $\no_1\colon (\cdistr(\R^\dimens))^\dimens\rightarrow (\distr(\R^\dimens))^\dimens$ by duality and the convolution formula holds in the sense of distributions. One has the reconstruction formula for the solenoidal part~$\sol{F}$ in the solenoidal decomposition $F=\sol{F}+\der\phi$ (see for example~\cite{SHA-integral-geometry-tensor-fields, SU:microlocal-analysis-integral-geometry})
\begin{equation}
\label{eq:invertingsolenoidalpartfromnormaloperator}
\sol{F}=c_{1, \dimens}(-\Delta)^{1/2}\no_1 F
\end{equation}
where~$c_{1, \dimens}$ is a constant depending on dimension and~$(-\Delta)^{1/2}$ operates componentwise on~$\no_1 F$. Both~$\xrt_1$ and~$\no_1$ are also defined for vector fields $F\in(\schwartz(\R^n))^n$.

\section{Partial differential operators and admissible functions}
\label{sec:condition}
In this section we introduce constant coefficient partial differential operators and also study the space of admissible functions~$\adm_V$ in more detail. A comprehensive treatment of constant coefficient partial differential operators can be found in H\"ormander's book~\cite{HO:analysis-of-pdos2}.

Let us denote by~$\pdos$ the set of all polynomials in~$\R^n$ with complex coefficients excluding the zero polynomial $P\equiv 0$.
A polynomial $P\in\pdos$ of degree $m\in\N$ can be identified with the constant coefficient partial differential operator~$P(D)$ of order $m\in \N$ as 
\begin{equation}
P(D)=\sum_{|\alpha|\leq m}a_\alpha D^\alpha, \quad a_\alpha\in\C ,
\end{equation}
where $D^\alpha=D^{\alpha_1}_1\cdots D^{\alpha_n}_n$, $D_j=-i\partial_j$ and $\alpha= (\alpha_1, \dotso , \alpha_n)\in\N^n$ is a multi-index such that $\abs{\alpha}=\alpha_1+\dotso +\alpha_n$. In fact, using the Fourier transform one sees that
\begin{align}
\widehat{P(D)}=P(\xi)=\sum_{|\alpha|\leq m}a_\alpha \xi^\alpha
\end{align}
where $\xi\in\R^n$ and $\xi^\alpha=\xi^{\alpha_1}\cdots\xi^{\alpha_n}$. The polynomial~$P(\xi)$ is also known as the full symbol of~$P(D)$. If $g\in\distr(\Omega)$ where $\Omega\subset\R^n$ is an open set, then one can define the distributional derivative $P(D)g\in\distr(\Omega)$ by duality. Further, it holds that $P(D)\colon H^r(\Omega)\rightarrow H^{r-m}(\Omega)$ is continuous with respect to the quotient norm~\cite[Theorem 12.15]{MI:distribution-theory2} (see equation~\eqref{eq:quotientnorm}).

The set of admissible functions~$\adm_V$ which we use in our main theorems can be written as the union
\begin{align}
\label{eq:admissibleunion}
\adm_V=\bigcup_{\substack{P\in\pdos \\ r\in \R}}\mathcal{H}^r_{P, V}(\R^n)
\end{align}
where $V\subset\R^n$ is some nonempty open set and $\mathcal{H}^r_{P, V}(\R^n)=\{f\in H^r(\R^n): P(D)f|_V=0\}$. We note that $\adm_V\subset H^{-\infty}(\R^n)$. The following proposition implies that the sets~$\mathcal{H}^r_{P, V}(\R^n)$ in the union~\eqref{eq:admissibleunion} are also Hilbert spaces.

\begin{proposition}
The subset $\mathcal{H}^r_{P, V}(\R^n)\subset H^r(\R^n)$ is a separable Hilbert space for all $r\in\R$, $P\in\pdos$ and nonempty open set $V\subset\R^n$.
\end{proposition}

\begin{proof}
Clearly $\mathcal{H}^r_{P, V}(\R^n)$ is a linear subspace of~$H^r(\R^n)$. Let $f_k\in \mathcal{H}_{P, V}^r(\R^n)$ be a sequence such that $f_k\rightarrow f$ in $H^r(\R^n)$. Then by the continuity of the restriction map $|_V\colon H^r(\R^n)\rightarrow H^r(V)$ we have that $f_k|_V\rightarrow f|_V$ in $H^r(V)$. From the continuity of~$P(D)\colon H^r(V)\rightarrow H^{r-m}(V)$ we obtain that $0=P(D)f_k|_V\rightarrow P(D)f|_V$ in~$H^{r-m}(V)$, implying that $f\in\mathcal{H}^r_{P, V}(\R^n)$. Therefore~$\mathcal{H}_{P, V}^r(\R^n)$ is a closed subspace of the separable Hilbert space~$H^r(\R^n)$ and hence itself a separable Hilbert space.
\end{proof}

\begin{remark}
We note that in the smooth case we have that $\smooth_{P, V}(\R^n)=\{f\in\smooth(\R^n):P(D)f|_V=0\}\subset\smooth(\R^n)$ is a closed subspace of~$\smooth(\R^n)$ and hence a Fr\'echet space. More generally, $\distr_{P, V}(\R^n)=\{f\in\distr(\R^n): P(D)f|_V=0\}\subset\distr(\R^n)$ is sequentially closed in~$\distr(\R^n)$ under the weak$^*$ convergence. These two facts follow from the continuity of $P(D)\colon\smooth(\R^n)\rightarrow\smooth(\R^n)$ and $P(D)\colon\distr(\R^n)\rightarrow\distr(\R^n)$ with respect to the standard topologies. More topological properties of kernels of constant coefficient partial differential operators can be found in~\cite{WE-properties-kernels-pdos}.
\end{remark}

\begin{remark}
\label{remark:condition}
The interpretation of the condition $f\in\adm_V$ is the following. If $f\in\adm_V$, then there is some $r\in\R$ and some~$P\in\pdos$ such that $f\in H^r(\R^n)$ and $P(D)f|_V=0$. The distributional derivatives commute with restrictions, i.e. $P(D)f|_V=P(D)(f|_V)$ where $f|_V\in\distr(V)$. Since $f\in H^r(\R^n)$ we see that~$f|_V$ is not only a distribution but in addition $f|_V\in H^r(V)$ for some $r\in\R$. Therefore the existence of $r\in\R$ and $P\in\pdos$ for which $P(D)f|_V=0$ means that $f|_V\in H^r(V)$ and~$f|_V$ is a weak solution to some homogeneous constant coefficient partial differential equation. In other words, $f|_V$ satisfies
\begin{equation}
\sum_{|\alpha|\leq m}a_\alpha D^\alpha (f|_V)=0, \quad f|_V\in\bigcup_{r\in\R}H^r(V),
\end{equation}
for some coefficients $a_\alpha\in\C$ and some integer $m\in\N$.
\end{remark}

The following proposition is important in the uniqueness of the partial data problem.

\begin{proposition}
\label{prop:vectorspace}
The set $\adm_V\subset H^{-\infty}(\R^n)$ is a vector space for every nonempty open set $V\subset\R^n$.
\end{proposition}

\begin{proof}
Let $f_1, f_2\in\adm_V$ and $\lambda\in\C$. This means that $f_1\in H^{r_1}(\R^n)$, $f_2\in H^{r_2}(\R^n)$ and $P_1(D)f_1|_V=P_2(D)f_2|_V=0$ for some $r_1, r_2\in\R$ and $P_1, P_2\in\pdos$. It follows that $f_1+\lambda f_2\in H^r(\R^n)$ where $r=\min\{r_1, r_2\}$ since the spaces $H^t(\R^n)$, $t\in\R$, are nested vector spaces. We also have that $P_1(D)P_2(D)(f_1+\lambda f_2)|_V=0$ since the distributional derivatives commute $P_1(D)P_2(D)=P_2(D)P_1(D)$. This implies that $f_1+\lambda f_2\in\adm_V$, i.e.~$\adm_V$ is a linear subspace of the vector space~$H^{-\infty}(\R^n)\subset\tempered(\R^n)$.
\end{proof}

\begin{remark}
\label{remark:uniquenesspartialdata}
The vector space structure of~$\adm_V$ is important since it implies that the partial data results we have proved in this article are indeed uniqueness results. Namely, if $f_1, f_2\in\cdistr(\R^n)\cap\adm_V$ (or $f_1, f_2\in C_\infty(\R^n)\cap\adm_V$) such that $\xrt_0f_1=\xrt_0 f_2$ on all lines intersecting~$V$, then $f_1-f_2\in\cdistr(\R^n)\cap\adm_V$ (or $f_1-f_2\in C_\infty(\R^n)\cap\adm_V$) and $\xrt_0(f_1-f_2)=0$ on all lines intersecting~$V$.  
Theorem~\ref{thm:xrayscalarpolynomial} then implies that $f_1-f_2=0$, i.e. the solution to the partial data problem is unique.
\end{remark}   

\section{Proofs of the main theorems}
\label{sec:proofs}
In this section we prove our main theorems. We need a few auxiliary results. The first one is a unique continuation result for fractional Laplacians and the second one is the Poincar\'e lemma for compactly supported vector-valued distributions.

\begin{lemma}[{\cite[Theorem 1.1]{CMR-ucp-higher-order-laplacians}}]
\label{lemma:uniquecontinuationoffractionallaplacian}
Let $\dimens \geq 1$, $s\in (-n/4,\infty)\setminus \Z$ and $u\in H^t(\R^\dimens)$ where $t\in\R$. If $(-\Delta)^s u|_V=0$ and $u|_V=0$ for some nonempty open set $V\subset\R^\dimens$, then $u=0$. The claim holds also for $s\in (-n/2, -n/4]\setminus\Z$ if $u\in\rapidly(\R^\dimens)$.
\end{lemma}

\begin{lemma}[Poincar\'e lemma]
\label{lemma:poincarelemma}
Let $U\in(\cdistr (\R^\dimens))^
\dimens$ such that $\der U=0$.
Then there is $\phi\in\cdistr (\R^\dimens)$ such that $U=\der\phi$.
\end{lemma}

The proof of lemma~\ref{lemma:poincarelemma} can be found for example in~\cite{HO-topological-vector-spaces, MA-poincare-derham-theorems}. The third lemma is a known result about the zero set of multivariate polynomials.

\begin{lemma}[{\cite[Lemma on p.1]{OKA-zerosofpolynomial}}]
\label{lemma:zerosetpolynomial}
Let $Q=Q(x)$ be a non-zero multivariate polynomial of order $m\in\N$
\begin{equation}
Q(x)=\sum_{|\alpha|\leq m}b_\alpha x^\alpha=\sum_{|\alpha|\leq m}b_\alpha x_1^{\alpha_1}\cdots x_n^{\alpha_n}, \quad b_\alpha\in\C,
\end{equation}
where $\alpha=(\alpha_1, \dotso, \alpha_n)\in\N^n$ is a multi-index such that $\abs{\alpha}=\alpha_1+\dotso+\alpha_n$. Then the set $Z_Q=\{x\in\R^n:Q(x)=0\}$ has Lebesgue measure zero.
\end{lemma}
Lemma~\ref{lemma:zerosetpolynomial} is proved in~\cite{OKA-zerosofpolynomial} for real coefficients but the result holds also for complex coefficients by splitting~$b_\alpha\in\C$ to its real and imaginary parts. We note that the set~$Z_Q$ is Zariski closed but not the whole space~$\R^n$. From the coarseness of the Zariski topology (i.e. there are relatively few closed sets) one can already deduce that the set~$Z_Q$ must be small in topological sense (see e.g.~\cite[Chapter 15.2]{DF-abstract-algebra}).

The next lemma shows how the normal operator of the X-ray transform of vector fields is related to the normal operator of scalar fields (see also~\cite[Proof of theorem 1.1]{IM-one-forms-partial-data}).

\begin{lemma}
\label{lemma:scalarvectornormaloperator}
Let $F\in(\cdistr(\R^n))^n$. Then $\der(\no_1 F)=(n-1)^{-1}\no_0(\der F)$ holds componentwise where~$\no_0$ acts on the components $(\der F)_{ij}\in\cdistr(\R^n)$.
\end{lemma}

\begin{proof}
The normal operator has the expression
\begin{equation}
(\no_1 F)_i
=
\sum_{j=1}^\dimens\frac{2x_ix_j}{\abs{x}^{\dimens+1}}\ast F_j.
\end{equation}
Rewrite the kernel as
\begin{equation}
\frac{2x_ix_j}{\abs{x}^{n+1}}
=
\frac{2}{\dimens-1}\bigg(\delta_{ij}\abs{x}^{1-\dimens}-\partial_i(x_j\abs{x}^{1-\dimens})\bigg)
\end{equation}
which implies that
\begin{equation}
(\no_1 F)_i
=
\frac{2}{\dimens-1}\bigg(
\frac{1}{2}\no_0 F_i
-
\sum_{j=1}^\dimens x_j\abs{x}^{1-\dimens}\ast\partial_i F_j
\bigg)
.
\end{equation}
Calculating the components of $\der(\no_1 F)$ we obtain
\begin{equation}
\label{eq:relationofnormaloperators}
\partial_k (\no_1F)_i-\partial_i(\no_1 F)_k
=
\frac{1}{\dimens-1}
\no_0(\partial_k F_i-\partial_i F_k)
.
\end{equation}
This means that $
\der(\no_1F)
=
(\dimens-1)^{-1}
\no_0(\der F)
$
where~$\no_0$ acts componentwise on $\der F$, giving the claim
\end{proof}

Now we are ready to prove our results. We start with the main theorem.

\begin{proof}[Proof of theorem \ref{thm:ucppolynomial}]
Let $f\in\adm_V$ and $s\in (-n/4, \infty)\setminus\Z$.
This means that $f\in H^r(\R^n)$ for some $r\in\R$ and $P(D)f|_V=0$ for some constant coefficient partial differential operator~$P(D)$ of order $m\in\N$ and nonempty open set~$V\subset\R^n$.
In particular, we have $f\in\tempered(\R^n)$ such that $\hat{f}=\langle\cdot\rangle^{-r} g$ where $g\in L^2(\R^n)$ and hence $\hat{f}\in L^1_{loc}(\R^n)$ is a locally integrable function.
Using the properties of the Fourier transform we see that $P(D)(\fraclaplace f)=\fraclaplace (P(D)f)$ because $P(D)$ has constant coefficients.
Since $P(D)$ is a local operator we obtain the conditions $P(D)f|_V=\fraclaplace(P(D)f)|_V=0$.
Now $P(D)\colon H^r(\R^n)\rightarrow H^{r-m}(\R^n)$ is continuous (see e.g.~\cite[Theorem 12.7]{MI:distribution-theory2}) and we have $P(D)f\in H^{r-m}(\R^n)$.
We can use lemma~\ref{lemma:uniquecontinuationoffractionallaplacian} for $P(D)f$ to obtain that $P(D)f=0$ as a tempered distribution.
Taking the Fourier transform this is equivalent to that $P(\xi)\hat{f}(\xi)=0$ almost everywhere where $P(\xi)$ is a multivariate polynomial of order~$m\in\N$.
Since $P(\xi)\neq 0$ almost everywhere by lemma~\ref{lemma:zerosetpolynomial}, we have that $\hat{f}=0$ almost everywhere and so $f=0$ as claimed.

Let then $f\in\rapidly(\R^n)\cap\adm_V$ and $s\in (-n/2, \infty)\setminus\Z$. Using the same arguments as above we obtain that $P(D)f|_V=\fraclaplace(P(D)f)|_V=0$ for some constant coefficient partial differential operator~$P(D)$ and nonempty open set $V\subset\R^n$.
We know that $f\in\rapidly(\R^n)$ is equivalent to that $\hat{f}\in\slowly(\R^n)$.
It follows from the Leibnitz product rule for multivariable functions that $\fourier (P(D)f)(\xi)=P(\xi)\hat{f}(\xi)\in\slowly(\R^n)$ since $P(\xi)$ is polynomial and the derivatives of~$\hat{f}$ are polynomially growing.
This is equivalent to that $P(D)f\in\rapidly(\R^n)$ and we can use lemma~\ref{lemma:uniquecontinuationoffractionallaplacian} to deduce that $P(D)f=0$ as a tempered distribution.
The rest of the proof is completed as above using the Fourier transform and the fact that $P(\xi)\neq 0$ almost everywhere.
\end{proof}

\begin{remark}
In the proof of theorem~\ref{thm:ucppolynomial} we used the fact that $f\in H^r(\R^n)$ implies that $\hat{f}\in L^1_{loc}(\R^n)$ is a locally integrable function.
For example, if $P(D)=-\Delta$, then $P(D)f=0$ implies $\abs{\xi}^2\hat{f}(\xi)=0$ and hence $\spt(\hat{f})\subset\{0\}$.
This means that~$\hat{f}$ is a linear combination of derivatives of the delta distribution, but there is no such non-zero combination in~$L^1_{loc}(\R^n)$.
In other words, a non-zero function in~$H^{-\infty}(\R^n)$ cannot be a polynomial on the whole space~$\R^n$, which for a tempered distribution is equivalent with the Fourier transform being supported at the origin.
The restriction of $f\in H^{-\infty}(\R^n)$ to an open set $V\subset\R^n$ can be a polynomial; compare this to the examples given in section~\ref{sec:admissible}.
\end{remark}


The rest of the results are then direct consequences of theorem~\ref{thm:ucppolynomial}.

\begin{proof}[Proof of theorem~\ref{thm:strongerucp}]
By the assumptions we have that $f\in\cdistr(\R^n)$ satisfies $P(D)f|_V=0$ for some constant coefficient partial differential operator $P(D)$ and $\partial^\beta(\fraclaplace f)(x_0)=0$ for some $x_0\in V$ and all $\beta\in\N^n$. Since all the derivatives of~$\fraclaplace f$ vanish at~$x_0$ and the partial derivatives and fractional Laplacian commute, we obtain that
\begin{equation}
0=(P(D)\partial^\beta)(\fraclaplace f)(x_0)=\partial^\beta(\fraclaplace (P(D)f))(x_0).
\end{equation}
Now $P(D)f\in\cdistr(\R^n)$ and we can use corollary 4 on page 652 in~\cite{CMR-ucp-higher-order-laplacians} to obtain that $P(D)f=0$. The claim then follows as in the proof of theorem~\ref{thm:ucppolynomial}.
\end{proof}

\begin{proof}[Proof of theorem \ref{thm:ucpnormaloperator}]
If $f\in\cdistr(\R^n)\cap\adm_V$ or $f\in C_\infty(\R^n)\cap\adm_V$, then also $f\in\rapidly(\R^n)\cap\adm_V$. Since $\no_0=(-\Delta)^{-1/2}$ up to a constant factor and $n\geq 2$ we have that $-1/2\in (-n/2, \infty)\setminus\Z$ and we can use theorem~\ref{thm:ucppolynomial} to obtain that $f=0$.
\end{proof}

\begin{proof}[Proof of theorem \ref{thm:xrayscalarpolynomial}]
The assumption $\xrt_0 f=0$ on all lines intersecting~$V$ implies that $\no_0 f|_V=0$. Since we also assume that $f\in\cdistr(\R^n)\cap\adm_V$ or $f\in C_\infty(\R^n)\cap\adm_V$ we obtain $f=0$ by theorem~\ref{thm:ucpnormaloperator}.
\end{proof}

\begin{proof}[Proof of theorem~\ref{thm:ucpnormalvector}]
By lemma~\ref{lemma:scalarvectornormaloperator} we have $\der(\no_1 F)=\no_0(\der F)$ componentwise up to a constant factor. Therefore $\partial^\beta(\no_0(\der F)_{ij})(x_0)=0$ for some $x_0\in V$, all $\beta\in\N^n$ and all $i, j=1, \dotso , n$. Now by locality of the exterior derivative $(\der F)_{ij}\in\cdistr(\R^n)\cap\adm_V$. Since $\no_0=(-\Delta)^{-1/2}$ up to a constant factor we can use theorem~\ref{thm:strongerucp} for the components $(\der F)_{ij}$ to obtain that $\der F=0$. Finally lemma~\ref{lemma:poincarelemma} implies that $F=\der\phi$ for some $\phi\in\cdistr(\R^n)$.
\end{proof}

\begin{proof}[Proof of theorem \ref{thm:xrayvectorpolynomial}]
The assumption $\xrt_1 F=0$ on all lines intersecting~$V$ implies that $\no_1 F|_V=0$. Especially $\der(\no_1 F)$ vanishes to infinite order at some point in~$V$ and we can use theorem~\ref{thm:ucpnormalvector} to deduce that $F=\der\phi$ for some $\phi\in\cdistr(\R^n)$. 
\end{proof}


\subsection*{Acknowledgements}
J.I. was supported by Academy of Finland (grants 332890 and 336254).
K.M. was supported by Academy of Finland (Centre of Excellence in Inverse Modelling and Imaging, grant numbers 284715 and 309963).
We want to thank the anonymous referees for their valuable feedback.

\bibliography{refs} 

\begin{thebibliography}{10}

\bibitem{BGU-lower-order-nonlocal-perturbations}
S.~{Bhattacharyya}, T.~{Ghosh}, and G.~{Uhlmann}.
\newblock {Inverse problem for fractional-Laplacian with lower order non-local
  perturbations}.
\newblock {\em Trans. Amer. Math. Soc.}, 374(5):3053--3075, 2021.

\bibitem{CLR18}
M.~Ceki\'{c}, Y.-H. Lin, and A.~R\"{u}land.
\newblock The {C}alder\'{o}n problem for the fractional {S}chr\"{o}dinger
  equation with drift.
\newblock {\em Calc. Var. Partial Differential Equations}, 59(3):Paper No. 91,
  46, 2020.

\bibitem{CWHM-sobolev-spaces-on-non-lipchtiz-domains}
S.~N. Chandler-Wilde, D.~P. Hewett, and A.~Moiola.
\newblock Sobolev spaces on non-{L}ipschitz subsets of {$\R^n$} with
  application to boundary integral equations on fractal screens.
\newblock {\em Integral Equations Operator Theory}, 87(2):179--224, 2017.

\bibitem{CNDK-solving-interior-problem-ct-with-apriori-knowledge}
M.~Courdurier, F.~Noo, M.~Defrise, and H.~Kudo.
\newblock Solving the interior problem of computed tomography using \textit{a
  priori} knowledge.
\newblock {\em Inverse Problems}, 24(6):065001, 2008.

\bibitem{CO-magnetic-fractional-schrodinger}
G.~Covi.
\newblock An inverse problem for the fractional {S}chr\"{o}dinger equation in a
  magnetic field.
\newblock {\em Inverse Problems}, 36(4):045004, 24, 2020.

\bibitem{CMR-ucp-higher-order-laplacians}
G.~Covi, K.~M{\"o}nkk{\"o}nen, and J.~Railo.
\newblock Unique continuation property and {P}oincar{\'e} inequality for higher
  order fractional {L}aplacians with applications in inverse problems.
\newblock {\em Inverse Probl. Imaging}, 15(4):641--681, 2021.

\bibitem{CMRU-higher-order-fractional-perturbations}
G.~{Covi}, K.~{M{\"o}nkk{\"o}nen}, J.~{Railo}, and G.~{Uhlmann}.
\newblock {The higher order fractional Calder{\'o}n problem for linear local
  operators: uniqueness}.
\newblock 2020.
\newblock arXiv:2008.10227.

\bibitem{deRham-differentiable-manifolds}
G.~{de Rham}.
\newblock {\em {Differentiable Manifolds}}.
\newblock Grundlehren der mathematischen Wissenschaften. {Springer-Verlag
  Berlin Heidelberg}, first edition, 1984.

\bibitem{DEN-inversion-of-3d-tensor-fields}
A.~Denisjuk.
\newblock {Inversion of the x-ray transform for 3D symmetric tensor fields with
  sources on a curve}.
\newblock {\em Inverse Problems}, 22(2):399--411, 2006.

\bibitem{DF-abstract-algebra}
D.~S. Dummit and R.~M. Foote.
\newblock {\em {Abstract Algebra}}.
\newblock Wiley, 2003.

\bibitem{FF-unique-continuation-fractional-ellliptic-equations}
M.~M. Fall and V.~Felli.
\newblock Unique continuation property and local asymptotics of solutions to
  fractional elliptic equations.
\newblock {\em Comm. Partial Differential Equations}, 39(2):354--397, 2014.

\bibitem{FF-unique-continuation-higher-laplacian}
V.~Felli and A.~Ferrero.
\newblock Unique continuation principles for a higher order fractional
  {L}aplace equation.
\newblock {\em Nonlinearity}, 33(8):4133--4191, 2020.

\bibitem{FST-partial-inversion-2D}
H.~Fujiwara, K.~Sadiq, and A.~Tamasan.
\newblock {Partial inversion of the 2D attenuated $X$-ray transform with data
  on an arc}.
\newblock {\em Inverse Probl. Imaging}, 2021.
\newblock Published online.

\bibitem{GR-fractional-laplacian-strong-unique-continuation}
M.~A. Garc\'{\i}a-Ferrero and A.~R\"{u}land.
\newblock Strong unique continuation for the higher order fractional
  {L}aplacian.
\newblock {\em Math. Eng.}, 1(4):715--774, 2019.

\bibitem{GSU-calderon-problem-fractional-schrodinger}
T.~Ghosh, M.~Salo, and G.~Uhlmann.
\newblock The {C}alder\'{o}n problem for the fractional {S}chr\"{o}dinger
  equation.
\newblock {\em Anal. PDE}, 13(2):455--475, 2020.

\bibitem{GR-distributions-and-operators}
G.~Grubb.
\newblock {\em {Distributions and Operators}}.
\newblock Graduate Texts in Mathematics. Springer-Verlag New York, first
  edition, 2009.

\bibitem{HE:integral-geometry-radon-transforms}
S.~Helgason.
\newblock {\em Integral {G}eometry and {R}adon transforms}.
\newblock Springer, New York, 2011.

\bibitem{HO:analysis-of-pdos}
L.~H\"ormander.
\newblock {\em The {A}nalysis of {L}inear {P}artial {D}ifferential {O}perators
  {I}. Distribution Theory and Fourier Analysis}.
\newblock Classics in Mathematics. Springer-Verlag Berlin Heidelberg, second
  edition, 2003.
\newblock Reprint of the 2nd edition 1990.

\bibitem{HO:analysis-of-pdos2}
L.~H\"ormander.
\newblock {\em The {A}nalysis of {L}inear {P}artial {D}ifferential {O}perators
  {II}. Differential Operators with Constant Coefficients}.
\newblock Classics in Mathematics. Springer-Verlag Berlin Heidelberg, first
  edition, 2005.
\newblock Reprint of the 1983 Edition (Grundlehren der mathematischen
  Wissenschaften Vol. 257).

\bibitem{HO-topological-vector-spaces}
J.~Horv\'{a}th.
\newblock {\em Topological {V}ector {S}paces and {D}istributions. {V}ol. {I}}.
\newblock Addison-Wesley Publishing Co., Reading, Mass.-London-Don Mills, Ont.,
  1966.

\bibitem{IM-unique-continuation-riesz-potential}
J.~{Ilmavirta} and K.~{M\"onkk\"onen}.
\newblock Unique continuation of the normal operator of the x-ray transform and
  applications in geophysics.
\newblock {\em Inverse Problems}, 36(4):045014, 2020.

\bibitem{IM-one-forms-partial-data}
J.~{Ilmavirta} and K.~{M{\"o}nkk{\"o}nen}.
\newblock {X-ray Tomography of One-forms with Partial Data}.
\newblock {\em SIAM J. Math. Anal.}, 53(3):3002--3015, 2021.

\bibitem{JUH-principles-of-doppler-tomography}
P.~Juhlin.
\newblock {Principles of Doppler Tomography}.
\newblock Technical report, {Center for Mathematical Sciences, Lund Institute
  of Technology}, S-221 00 Lund, Sweden, 1992.

\bibitem{KKW-stability-of-interior-problems}
E.~Katsevich, A.~Katsevich, and G.~Wang.
\newblock Stability of the interior problem with polynomial attenuation in the
  region of interest.
\newblock {\em Inverse Problems}, 28(6):065022, 2012.

\bibitem{KEQ-wavelet-methods-ROI-tomography}
E.~Klann, E.~T. Quinto, and R.~Ramlau.
\newblock {Wavelet methods for a weighted sparsity penalty for region of
  interest tomography}.
\newblock {\em Inverse Problems}, 31(2):025001, 2015.

\bibitem{KLM-on-local-tomography}
P.~Kuchment, K.~Lancaster, and L.~Mogilevskaya.
\newblock On local tomography.
\newblock {\em Inverse Problems}, 11(3):571--589, 1995.

\bibitem{MA-poincare-derham-theorems}
S.~Mardare.
\newblock {On Poincar{\'e} and de Rham's theorems}.
\newblock {\em Rev. Roumaine Math. Pures Appl.}, 53(5-6):523--541, 2008.

\bibitem{ML-strongly-elliptic-systems}
W.~McLean.
\newblock {\em Strongly {E}lliptic {S}ystems and {B}oundary {I}ntegral
  {E}quations}.
\newblock Cambridge University Press, Cambridge, 2000.

\bibitem{MI:distribution-theory2}
D.~Mitrea.
\newblock {\em Distributions, {P}artial {D}ifferential {E}quations, and
  {H}armonic {A}nalysis}.
\newblock Universitext. Springer International Publishing, 2nd edition, 2018.

\bibitem{NA-mathematics-computerized-tomography}
F.~Natterer.
\newblock {\em The {M}athematics of {C}omputerized {T}omography}, volume~32 of
  {\em Classics in Applied Mathematics}.
\newblock Society for Industrial and Applied Mathematics (SIAM), Philadelphia,
  PA, 2001.
\newblock Reprint of the 1986 original.

\bibitem{NO-tomographic-recostruction-of-vector-fields}
S.~J. Norton.
\newblock {Tomographic Reconstruction of 2-D Vector Fields: Application to Flow
  Imaging}.
\newblock {\em Geophys. J. Int.}, 97(1):161--168, 1989.

\bibitem{NO-unique-tomographic-reconstruction-doppler}
S.~J. Norton.
\newblock {Unique Tomographic Reconstruction of Vector Fields Using Boundary
  Data}.
\newblock {\em IEEE Trans. Image Process.}, 1(3):406--412, 1992.

\bibitem{OKA-zerosofpolynomial}
M.~Okamoto.
\newblock {Distinctness of the Eigenvalues of a Quadratic form in a
  Multivariate Sample}.
\newblock {\em Ann. Statist.}, 1(4):763--765, 1973.

\bibitem{PSU-tensor-tomography-progress}
G.~P. Paternain, M.~Salo, and G.~Uhlmann.
\newblock {Tensor tomography: Progress and challenges}.
\newblock {\em Chin. Ann. Math. Ser. B}, 35(3):399--428, 2014.

\bibitem{QU-singularities-x-ray-transform-limited-data}
E.~Quinto.
\newblock {Singularities of the X-Ray Transform and Limited Data Tomography in
  $\mathbb{R}^2$ and $\mathbb{R}^3$}.
\newblock {\em SIAM J. Math. Anal.}, 24(5):1215--1225, 1993.

\bibitem{QU-artifacts-and-singularities-limited-tomography}
E.~Quinto.
\newblock Artifacts and {V}isible {S}ingularities in {L}imited {D}ata {X}-{R}ay
  {T}omography.
\newblock {\em Sens. Imaging}, 18, 2017.

\bibitem{RA-microlocal-analysis-doppler-transform}
K.~Ramaseshan.
\newblock {Microlocal Analysis of the Doppler Transform on $\mathbb{R}^3$}.
\newblock {\em J. Fourier Anal. Appl.}, 10(1):73--82, 2004.

\bibitem{RU-unique-continuation-scrodinger-rough-potentials}
A.~R\"{u}land.
\newblock Unique continuation for fractional {S}chr\"{o}dinger equations with
  rough potentials.
\newblock {\em Comm. Partial Differential Equations}, 40(1):77--114, 2015.

\bibitem{SCHU-3d-doppler-transform-reconstruction-and-kernels}
T.~Schuster.
\newblock {The 3D Doppler transform: elementary properties and computation of
  reconstruction kernels}.
\newblock {\em Inverse Problems}, 16(3):701--722, 2000.

\bibitem{SCHU-importance-of-vector-field-tomography}
T.~Schuster.
\newblock {The importance of the Radon transform in vector field tomography}.
\newblock In R.~Ramlau and O.~Scherzer, editors, {\em The Radon Transform: The
  First 100 Years and Beyond}. de Gruyter, 2019.

\bibitem{SHA-vector-tomography-incomplete-data}
V.~Sharafutdinov.
\newblock {Slice-by-slice reconstruction algorithm for vector tomography with
  incomplete data}.
\newblock {\em Inverse Problems}, 23(6):2603--2627, 2007.

\bibitem{SHA-integral-geometry-tensor-fields}
V.~A. Sharafutdinov.
\newblock {\em Integral geometry of tensor fields}.
\newblock Inverse and Ill-posed Problems Series. VSP, Utrecht, 1994.

\bibitem{SU:microlocal-analysis-integral-geometry}
P.~{Stefanov} and G.~{Uhlmann}.
\newblock {\em Microlocal {A}nalysis and {I}ntegral {G}eometry (working
  title)}.
\newblock 2018.
\newblock Draft version.

\bibitem{TRE:topological-vector-spaces-distributions}
F.~Tr\`eves.
\newblock {\em Topological {V}ector {S}paces, {D}istributions and {K}ernels}.
\newblock Academic Press, New York-London, 1967.

\bibitem{VER-integral-geometry-symmetric-tensor-incomplete}
L.~B. Vertgeim.
\newblock Integral geometry problems for symmetric tensor fields with
  incomplete data.
\newblock {\em J. Inverse Ill-Posed Probl.}, 8(3):355--364, 2000.

\bibitem{WE-properties-kernels-pdos}
J.~Wengenroth.
\newblock {Topological properties of kernels of partial differential
  operators}.
\newblock {\em Rocky Mountain J. Math.}, 44(3):1037--1052, 2014.

\bibitem{YYJW-high-order-TV-minimization}
J.~Yang, H.~Yu, M.~Jiang, and G.~Wang.
\newblock High-order total variation minimization for interior tomography.
\newblock {\em Inverse Problems}, 26(3):035013, 2010.

\bibitem{YA-higher-order-laplacian}
R.~{Yang}.
\newblock {On higher order extensions for the fractional {L}aplacian}.
\newblock 2013.
\newblock arXiv:1302.4413.

\end{thebibliography}
\bibliographystyle{abbrv}

\end{document}